%% file: ArithmeticConvolution_v3.tex
\documentclass[reqno, a4paper
]{amsart} 

\usepackage[latin1]{inputenc}
\usepackage{subfigure}

\usepackage{amssymb} 
\usepackage{amsthm}
\usepackage{upref} 
\usepackage{graphicx}
\usepackage{color}
\usepackage{url}
\usepackage{paralist} 

\usepackage[english]{babel}

\usepackage{calc}

\usepackage{mathrsfs} 

\usepackage[all]{xy}

\usepackage{hyperref}

\hypersetup{pdfauthor={Spencer Backman, Matthias Lenz},pdftitle={}, 
  pdfkeywords={} }

\newtheorem{Theorem}{Theorem}
\newtheorem{Corollary}[Theorem]{Corollary}

\newtheorem{Lemma}[Theorem]{Lemma}

\theoremstyle{definition}
\newtheorem{Definition}[Theorem]{Definition}

\newtheorem{Example}[Theorem]{Example}

\theoremstyle{remark}
\newtheorem{Remark}[Theorem]{Remark}

\DeclareMathOperator{\Hom}{Hom}
\DeclareMathOperator{\DM}{DM}
\DeclareMathOperator{\relint}{relint}

\DeclareMathOperator{\vol}{{{vol}}}
\newcommand{\abs}[1]{\left|#1\right|}
\DeclareMathOperator{\rank}{rk}

\DeclareMathOperator{\lcm}{lcm} 

\newcommand{\st}{s.\,t.\ } 
\newcommand{\ie}{\textit{i.\,e.\ }} 
\newcommand{\eg}{\textit{e.\,g.\ }} 

\newcommand{\N}{\mathbb{N}}
\newcommand{\Z}{\mathbb{Z}}

\newcommand{\R}{\mathbb{R}}

\newcommand{\K}{\mathbb{K}}
\newcommand{\Acal}{\mathcal{A}}

\newcommand{\Dcal}{\mathcal{D}}

\newcommand{\Pcal}{\mathcal{P}}

\newcommand{\Vcal}{\mathcal{V}}

\newcommand{\Fcal}{\mathcal F}
\newcommand{\Gcal}{\mathcal G}

\newcommand{\tutte}{{\mathfrak T}} %
\newcommand{\aritutte}{{\mathfrak M}} %
\newcommand{\matroidclass}{{\mathbb M}} %

\author{Spencer Backman}
\address{%
Hausdorff Center for Mathematics, 53115 Bonn, Germany}
\email{spencerbackman@gmail.com}
\author{Matthias Lenz}
\address{Universit\'e de Fribourg, D\'epartement de Math\'ematiques, 1700 Fribourg, 
Switzerland}
\email{matthias.lenz@unifr.ch}

\thanks{The second author was supported by a Swiss Government Excellence Scholarship for Foreign Scholars 
 and subsequently by a
 fellowship within the postdoc program of the German Academic Exchange Service 
 (DAAD)}

\title[A convolution formula for arithmetic Tutte polynomials]
{%
A convolution formula for Tutte polynomials of  arithmetic matroids 
and other combinatorial structures
}

\date{\today}
\address{%
}


\begin{document}

\begin{abstract}
 In this note we generalize the convolution formula for the Tutte polynomial of
 Kook--Reiner--Stanton and Etienne--Las~Vergnas 
 to a more general setting that   
 includes both arithmetic matroids and delta-matroids.
 As   corollaries, we obtain new proofs of two positivity results
 for pseudo-arithmetic matroids and
  a combinatorial interpretation of the arithmetic Tutte polynomial at infinitely many points  in terms of arithmetic flows and colorings.
  We also exhibit connections with a decomposition of Dahmen--Micchelli spaces and lattice point counting in zonotopes.
\end{abstract}

\maketitle

\section{Introduction}

 Matroids are combinatorial structures that capture and abstract the notion of independence.
 They were introduced in the 1930s, and since then 
 they have become an important part
 of combinatorics and other areas of pure and applied mathematics.
 The Tutte polynomial is an important matroid invariant. Many invariants of graphs and hyperplane arrangements can be obtained 
 as specializations of the Tutte polynomial 
 \cite{brylawski-oxley-1992}.
 Kook--Reiner--Stanton  \cite{kook-reiner-stanton-1999} and Etienne--Las~Vergnas \cite{etienne-las-vergnas-1998}
 found a so-called convolution formula for the Tutte polynomial $\tutte_M$ of a matroid $M$:
 \begin{equation}
  \tutte_M(x,y) = \sum_{A\subseteq M}  \tutte_{M|_A}(0, y) \tutte_{M/A}(x,0).
 \end{equation}
 In this note  we will generalize this formula to the far more general setting of ranked sets with multiplicities.
 
A \emph{ranked set with multiplicities} is a finite set $M$, together with a rank function 
$\rank : 2^M \to \mathbb{Z}$
that satisfies $\rank(\emptyset)=0$ and a multiplicity function $m : 2^M \to R$, where $R$ denotes a commutative ring with $1$.

This setting contains the following combinatorial structures as special cases:
 \begin{itemize}
  \item \emph{Matroids}: if $\rank$ satisfies the rank axioms of a matroid, $R=\Z$,  and $m \equiv 1$ (\eg \cite{MatroidTheory-Oxley}).
  \item \emph{Pseudo-arithmetic matroids}: if $(M,\rank)$ is a matroid and $m : 2^M \to \R_{\ge 0}$ satisfies  certain positivity conditions %
  \cite{branden-moci-2014}.
  \item \emph{Quasi-arithmetic matroids}: if $(M,\rank)$ is a matroid and $m : 2^M \to \Z_{\ge 1}$ satisfies certain divisibility conditions \cite{branden-moci-2014}.
  \item \emph{Arithmetic matroids}: if $(M,\rank,m)$ is both a pseudo-arithmetic matroid and a quasi-arithmetic matroid  \cite{branden-moci-2014,moci-adderio-2013}.
  \item \emph{Integral polymatroids}: if $\rank$ is the submodular function that defines an integral polymatroid, $R=\Z$ and $m \equiv 1$ (\eg \cite[Chapter~44]{schrijver-co-volB})
  \item \emph{Rank functions of delta-matroids and ribbon graphs:} 
     one can choose $m\equiv 1$ and  $\rank=\rho$, the 
     rank function of an even delta-matroid $(M,\Fcal)$ in the sense of 
       Chun--Moffatt--Noble--Rueckriemen
     \cite{chun-moffat-noble-rueckriemen-2014,krajewski-moffatt-tanasa-2015}.
     Ribbon graphs \cite{bollobas-riordan-2002} define delta-matroids
     in a similar way as 
     graphs define matroids \cite{bouchet-1989,chun-moffat-noble-rueckriemen-2014}.
\end{itemize}
See Section~\ref{Section:Background} for definitions.
Sometimes, we will write $\rank_M$ and $m_M$ to denote the rank and multiplicity functions of $M$ and we will occasionally write $M$ instead of 
$(M,\rank_M, m_M)$
to denote the
ranked set with multiplicities.

We will show that the convolution formula of  Kook--Reiner--Stanton %
and Etienne--Las~Vergnas %
 holds in a very general setting. The only thing we require is that restriction and contraction are defined in the usual way:
let $A\subseteq M$. The \emph{restriction} $M|_A$ is the ranked set with multiplicities 
$(A,\rank|_A, m|_A)$, where  $\rank|_A$ and $m|_A$ denote the restrictions of $\rank$ and $m$ to $A$.
 The \emph{contraction} $M / A$ is the
 ranked set with multiplicities
 $(M\setminus A,\rank_{M/A}, m_{M/A})$, where
 $\rank_{M/A}(B):=\rank_M(B\cup A) - \rank_M(A)$ and
  $ m_{M/A}(B) := m_{M}(B\cup A)$ for $B\subseteq M\setminus A$.

To a ranked set with multiplicities, we  associate the \emph{arithmetic Tutte function}
\begin{equation}
   \aritutte_M(x,y) = \sum_{A\subseteq M} m(A) (x-1)^{\rank(M) - \rank(A)} (y-1)^{\abs{A}-\rank(A)} \in R(x,y) 
\end{equation}
 and the \emph{Tutte function} $\tutte_M(x,y) = \sum_{A\subseteq M}  (x-1)^{\rank(M)-\rank(A)} (y-1)^{\abs{A}-\rank(A)} \in R(x,y)$.
 As usual, $R(x,y)$ denotes the ring of rational functions in $x$ and $y$ with coefficients in $R$.
 Note that $\aritutte_M(x+1,y+1)$ and $\tutte_M(x+1,y+1)$ are Laurent polynomials in $R[x^{\pm 1},y^{\pm 1}]$.
 If $\rank(A)\le \rank(M)$ and $ \rank(A) \le \abs{A} $ for all $ A\subseteq M $, then both functions are polynomials in $R[x,y]$.

 If $M$ is a matroid, $\tutte_M(x,y)$ is the usual Tutte polynomial.
 As far as we know, the
 Tutte Laurent polynomial $\tutte_M(x+1, y+1)$ of a polymatroid $M$ has not been studied yet. 
 However, other
 Tutte invariants of polymatroids have appeared in the literature
 \cite{cameron-fink-2016,oxley-whittle-1993}.
 If $M$ is a (quasi/pseudo)-arithmetic matroid,  $\aritutte_M(x,y)$ is the usual
 arithmetic Tutte polynomial \cite{branden-moci-2014,moci-adderio-2013,moci-tutte-2012}.
 The arithmetic Tutte polynomial appears in many different contexts, \eg in the study of the  combinatorics and topology of toric arrangements, 
 of cell complexes, the 
 theory of vector partition functions,  and Ehrhart theory of zonotopes 
 \cite{bajo-burdick-chmutov-2014,moci-adderio-ehrhart-2012,lenz-arithmetic-2016,moci-tutte-2012,stanley-1991}.

  If $\rank$ is the rank function of an even delta-matroid
  in the sense of Chun--Moffatt--Noble--Rueckriemen
\cite{chun-moffat-noble-rueckriemen-2014,krajewski-moffatt-tanasa-2015}, then
  $\tutte_M$ is the 
  $2$-variable Bollob\'as--Riordan polynomial of the delta-matroid    
  (see \cite{chun-moffat-noble-rueckriemen-2014} or \cite[(42)]{krajewski-moffatt-tanasa-2015}).
   A special case is the $2$-variable Bollob\'as--Riordan polynomial 
  of a ribbon graph \cite[p.~22]{krajewski-moffatt-tanasa-2015}.

\smallskip
The following theorem is our main result.
\begin{Theorem}
  \label{Theorem:ArithmeticConvolution}
  Let $(M,\rank,m)$ be a ranked set with multiplicities  (\eg an arithmetic matroid). Let $\aritutte_M$ denote its arithmetic Tutte polynomial
  and let   $\tutte_M$ denote its Tutte polynomial.
  Then
  \begin{align}
   \aritutte_M(x,y) &= \sum_{A\subseteq M}  \aritutte_{M|_A}(0, y) \tutte_{M/A}(x,0) \\
   &= \sum_{A\subseteq M}  \tutte_{M|_A}(0, y) \aritutte_{M/A}(x,0).
   \end{align}
 \end{Theorem}
Kook, Reiner, and Stanton proved this result in the case where
$(M,\rank)$ is a matroid and $m\equiv 1$ \cite{kook-reiner-stanton-1999}.
Their result also follows easily from
a theorem of Etienne and Las~Vergnas on the decomposition 
of the ground set of a matroid that has a bijective proof \cite[Theorem~5.1]{etienne-las-vergnas-1998}.
It would be interesting to give a bijective proof of our result in the case of arithmetic 
matroids.
 The result of Kook--Reiner--Stanton can also be proved using Hopf algebras
\cite{duchamp-nguyen-krajewski-tanasa-2013,kook-reiner-stanton-1999}.

In the case of even delta-matroids,
our
 theorem specializes to a convolution formula for the $2$-variable Bollob\'as--Riordan polynomial
\cite[Theorem~16(2)]{krajewski-moffatt-tanasa-2015}.

 \smallskip
 Theorem~\ref{Theorem:ArithmeticConvolution}  provides a new method to prove that the coefficients of the Tutte polynomial of a pseudo-arithmetic matroid are positive
 (\cite[Theorem~5.1]{moci-adderio-2013} and \cite[Theorem~4.5]{branden-moci-2014}).
 \begin{Corollary} %
 \label{Corollary:PositivityOfCoefficients}
   The coefficients of the  Tutte polynomial of a pseudo-arithmetic matroid are positive integers.
 \end{Corollary}

 \begin{Remark}
 \label{Remark:toricvertexdecomposition}
 Let $M$ be an arithmetic matroid that is represented by a list of vectors $X$ in some finitely generated abelian group.
 Let $\Vcal(X)$ denote the set of vertices of the corresponding generalized toric arrangement (for definitions see \cite{moci-tutte-2012}).
   If we set $x=1$, the second expression for $\aritutte_M(x,y)$ in 
    Theorem~\ref{Theorem:ArithmeticConvolution} is equivalent to 
   \cite[Lemma~6.1]{moci-tutte-2012}, which states that
   \begin{equation}
   \label{eq:arituttevertexdecomposition}
   \aritutte_M(1,y) = \sum_{ p\in \Vcal(X) } \tutte_{M_p} (1,y).
   \end{equation}
   Here,
   $M_p$ denotes the matroid represented by the sublist of $X$ that consists of all elements that define a hypersurface that contains $p$.
   This equivalence is explained in more detail in Section~\ref{Section:Proofs}.
   
   \eqref{eq:arituttevertexdecomposition} is related to 
   two decomposition formulas 
   in the theory of splines and vector partition functions: 
   the decomposition of 
   the discrete  space $\DM(X)$ into continuous $\Dcal$-spaces 
   $\DM(X) = \bigoplus_{p \in \Vcal(X)} e_p \Dcal(X_p)$ 
   by Dahmen and Micchelli \cite{dahmen-micchelli-1985b} (see also \cite[Theorem 49]{BoxSplineBook} and \cite[(16.1)]{concini-procesi-book})
   and dually,  the decomposition of the periodic $\Pcal$-spaces by the second author \cite{lenz-arithmetic-2016}.
   These decompositions could be a step towards a bijective proof of our result.
 \end{Remark}

 For two multiplicity functions $m_1,m_2 : 2^M \to R$, we will consider their product $m_1m_2$, defined by 
 $(m_1m_2)(A):= m_1(A)m_2(A)$.
 The following generalization of our main theorem was suggested to us by Luca Moci.
 It can be proven in a similar way.
 A complete proof will appear in a future article. 
 \begin{Theorem}
 \label{Theorem:GeneralisedMainTheorem}
   Let $(M,\rank,m_1)$ and $(M,\rank,m_2)$ be two ranked sets with multiplicity.
   Then $(M,\rank,m_1  m_2)$ is a ranked set with multiplicity
   and its arithmetic Tutte polynomial is given by the convolution formula
   \begin{equation}
    \aritutte_{(M,\rank,m_1  m_2)}(x,y) 
      = \sum_{A\subseteq M}  \aritutte_{ (M,\rank,m_1)   |_A}(0, y) \aritutte_{(M,\rank,m_2)   /A}(x,0).
   \end{equation}
 \end{Theorem}
 Theorem~\ref{Theorem:GeneralisedMainTheorem}  implies a generalized version of the key lemma  (Lemma~2) of \cite{delucchi-moci-2016}.
\begin{Corollary}[Positivity of products of multiplicity functions]
\label{Corollary:PforProduct}
 Let  $(M, \rank)$ be a matroid
 and let  $m_1, m_2 : 2^E \to \R$ be two functions.

 If both $m_1$ and $m_2$ satisfy the positivity axiom (cf.~\eqref{eq:Paxiom}), so does their product $m_1m_2$.
\end{Corollary}

\begin{Remark}
    Delucchi and Moci \cite{delucchi-moci-2016} 
    remarked that Corollary~\ref{Corollary:PforProduct} implies that
    if both $(E, \rank,m_1)$ and $(E, \rank,m_2)$ are arithmetic matroids, then 
    $(E, \rank,m_1 m_2)$ is an arithmetic matroid as well. They used this to answer a question of Bajo--Burdick--Chmutov
    on cellular matroids of CW complexes \cite{bajo-burdick-chmutov-2014}.

Note that $(E, \rank,m_1 m_2)$ is not necessarily representable, even if both $(E, \rank,m_1)$ are representable.
As an example, consider the arithmetic matroid $(E,\rank, m)$ represented by the list of vectors
 $X = ((1,0),(0,1),(1,1),(1,-1))$ and the  arithmetic matroid $(E,\rank, m^2)$. 
 The underlying matroid is uniform in both cases. 
 Suppose there is a list of vectors $X'$ that represents $(E,\rank, m^2)$.
 Since $m^2$ is equal to one on five of the six bases, one can assume without loss of generality that two of the vectors in $X'$
 are $(1,0)$ and $(0,1)$. Then it follows that the other two are of the form $(\pm 1, \pm 1)$. This implies that  
 all bases have multiplicity one or two, which is a contraction.
 Questions of this type are discussed in more detail in \cite{lenz-ppcram-2017}.
\end{Remark}

\subsection*{Zonotopes}
It is easy to see that the number of integer points in a polytope is equal to the sum of the number of integer points in the 
interior of all of its faces.
In the case of zonotopes, this statement is equivalent to the specialization
 of Theorem~\ref{Theorem:ArithmeticConvolution} to $(x,y)=(2,1)$.
\begin{Corollary}
\label{Corollary:TwoOne}
Let $X=(x_1,\ldots, x_N)\subseteq \Z^d$ be a list of  vectors and
let $Z(X):=\{ \sum_{i=1}^N \lambda_i x_i : 0\le \lambda_i \le 1\}$ be the
zonotope defined by $X$.
Then
\begin{align}
  \abs{Z(X)\cap \Z^d} = \aritutte(2,1) &= \sum_{A\subseteq X} \aritutte_{M|_A}(0,1)
  \tutte_{M/A}(2,0) 
  \\
  & =
   \sum_{X \supseteq A \text{ flat}} \aritutte_{M|_A}(0,1)  \tutte_{M/A}(2,0) 
  = \sum_{ F } \abs{ \relint(F)\cap \Z^d },
  \notag
\end{align}
where the last sum is over all faces of $Z(X)$.
\end{Corollary}

 Barvinok and Pommersheim proved a geometric convolution-like formula 
 for the number of integer points in a lattice zonotope. It would be interesting to find a connection with our convolution formula.
 \begin{Theorem}[{\cite[Section~7]{barvinok-pommersheim-1999}}]
      Let $X\subseteq \Z^d$ be a  list of $N$ vectors. Then
      \begin{equation}
       \abs{ Z(X) \cap \Z^d } =  \sum_F \vol(F) \gamma(P,F),  
      \end{equation}
 where the sum is over all faces $F$ of the zonotope  and $\gamma(P,F)$ 
 denotes the exterior angle of $F$ at $P$.
 The volume of a face is measured intrinsically with respect to the lattice. 
 
   More specifically, the $k$th coefficient of the Ehrhart polynomial  
   $E_X(q) = q^N \aritutte_X(1 + \frac 1q,1)$ of the zonotope
   is equal to  $ \sum_{F: \dim F=k}  \vol(F) \gamma(P,F) $.
 \end{Theorem}

 \subsection*{Flows and colorings}
 
 In this subsection we will give a combinatorial interpretation of the evaluation of the arithmetic Tutte polynomial
 and a closely related polynomial,  the modified Tutte--Krushkal--Renardy polynomial, at infinitely many integer values 
 in terms of arithmetic flows and colorings.  This works for arbitrary representable arithmetic matroids.

 D'Adderio and Moci defined a class of ``graphic arithmetic matroids'' using graphs whose edges are labeled by positive integers
 \cite{moci-adderio-flow-2013}.
 One can define so-called arithmetic flows and arithmetic colorings on these graphs.  These notions of flows and
 colorings were extended by Br\"and\'en and Moci to the setting where $X$ is a finite list of elements from a finitely generated abelian group \cite{branden-moci-2014}.
 These arithmetic flows and colorings are related to our convolution formula in a similar way as
 classical flows and colorings are related to the classical convolution formula \cite[Theorem~2]{kook-reiner-stanton-1999}.
 Arithmetic flows and colorings contain 
 flows and colorings of CW complexes \cite{beck-breuer-godkin-martin-2014,beck-kemper-2012}
 as a special case, 
 when the list of vectors is 
 taken to be a boundary operator of CW complexes  
 \cite[Lemma~4]{delucchi-moci-2016}.

We briefly review the setup of Br\"and\'en and Moci. 
Let $G$ be a finitely generated abelian group.  
Let $X$ 
be a finite list (or sequence) of elements of $G$. %
We call $\phi \in \Hom(G, \mathbb{Z}_q)$ a \emph{proper arithmetic $q$-coloring} if $\phi(x) \neq 0$ for all $x \in X$. 
We denote the number of proper arithmetic $q$-colorings of $X$ by $\chi_{X}(q)$.

A \emph{nowhere zero $q$-flow} on $X$ 
is a function $\psi : X \to \Z_q \setminus \{0\}$ \st
 $\sum_{x \in X} \psi(x) x = 0$ in $G/qG$.
 We denote the number of such functions by $\chi^*_{X}(q)$.

For $B\subseteq X$, let $G_B$ denote the torsion subgroup of the quotient
$G / \left\langle \{ x : x\in B \} \right\rangle$ and let $m(B):=\abs{G_B}$.

 Let $\lcm(X):=\lcm\{ m(B) : B \subseteq X \text{ basis}\}$.
 We define the following two subsets of the set of   positive integers:
\begin{align}
 \mathbb{Z}_M(X) &:= \{q\in \Z_{>0} : \gcd(q, \lcm(X))=1\} 
 \\
 \text{and }
 \mathbb{Z}_A(X) &:= \{q\in \Z_{>0} : q G_B = \{0\} \text{ for all bases } B \subseteq X \}.
\end{align}
Given a list of vectors $X$ with associated arithmetic matroid $(M,\rank,m)$ we let $\aritutte_X(x,y)$ 
denote the arithmetic Tutte polynomial $\aritutte_{(M,\rank,m)}(x,y)$. 
Furthermore, we let $\aritutte_{X^2}(x,y)$ denote the arithmetic Tutte polynomial $\aritutte_{(M,\rank,m^2)}(x,y)$. 
We recall that by Corollary~\ref{Corollary:PforProduct} (or by \cite{delucchi-moci-2016}),  $(M,\rank,m^2)$ is indeed an arithmetic matroid.
 The polynomial $\aritutte_{X^2}(x,y)$ has a special significance 
 for arithmetic matroids that arise from CW complexes.
 In this case, the \emph{modified $j$th Tutte--Krushkal--Renardy polynomial}, that was introduced in \cite{bajo-burdick-chmutov-2014}, 
 is equal to
 the arithmetic Tutte polynomial $\aritutte_{X^2}(x,y)$, where $X$ is the list of vectors obtained from the $j$th boundary operator
 \cite[Section~4]{delucchi-moci-2016}.
 In this setting, the modified $j$th Tutte--Krushkal--Renardy 
 polynomial can be recovered from %
 Corollary \ref{Corollary:squared} below.

\begin{Theorem}[Br\"and\'en--Moci, \cite{branden-moci-2014}]
\label{Theorem:BrandenMociFlowsColourings}
Let $G$ and $X$  be as above.
\begin{align}
\text{If $q \in \mathbb{Z}_A(X)$, then 
 } \chi_{X}(q) &= (-1)^{\rank(X)}q^{\rank(G) - \rank(X)} \aritutte_X(1-q,0)
\\
\text{and }
\chi^*_{X}(q) &= (-1)^{\abs{X} - \rank(X)} \aritutte_X(0,1-q).
\\
\text{If $q \in \mathbb{Z}_M(X)$, then 
}
\chi_{X}(q) &= (-1)^{\rank(X)}q^{\rank(G)-\rank(X)} \tutte_X(1-q,0) \\
\text{and } \chi^*_{X}(q) &= (-1)^{\abs{X}-\rank(X)} \tutte_X(0,1-q).
\end{align}
\end{Theorem}

\begin{Example}
  Let $X=((2,0),(-1,1),(1,1))$.
Then $\lcm(X)=2$,  
         $ \mathbb{Z}_M(X) = \{1,3,5,7,\ldots\} $, and
         $ \mathbb{Z}_A(X) = \{2,4,6,8,\ldots\} $.
	 The polynomials are
	 $\chi_X(q)|_{\mathbb{Z}_A(X)}= q^2 - 4q + 4$,
	 $\chi_X(q)|_{\mathbb{Z}_M(X)}= q^2 - 3q + 2$,
	 $\chi_X^*(q)|_{\mathbb{Z}_A(X)}= 2q-3$, and
	 $\chi_X^*(q)|_{\mathbb{Z}_M(X)}= q-1$.
	 Hence there are two proper arithmetic $3$-colorings ($[1,0]$ and $[2,0]$)
	 and two nowhere zero  $3$-flows ($[1,1,2]$ and $[2,2,1]$).
\end{Example}

Let $A\subseteq X$.
We denote the sublist of $X$ that is indexed by $A$ by $X|_A$ (restriction) 
and the projection of $X|_{X\setminus A}$ to $G/A := G/\langle \{x : x\in A\}\rangle$ by $X/A$ (contraction).
 
 \begin{Corollary}
 \label{Corollary:squared}
  Let $G$ and $X$ be as above and $p,q  \in \mathbb{Z}_A(X)$. Then  
\begin{equation}
     \aritutte_{X^2}(1-p,1-q) =
  p^{\rank(G)-\rank(X)}(-1)^{\rank(X)}\sum_{ A\subseteq X}(-1)^{|A|}\chi^*_{X|_A}(q) \chi_{X/A}(p).
\end{equation}
 \end{Corollary}

 \begin{Corollary}
 \label{Corollary:FlowsColourings}
  Let $G$ and $X$  be as above, $p \in \mathbb{Z}_A(X)$ and $q \in \mathbb{Z}_M(X)$ then  
\begin{equation}
\aritutte_{X}(1-p,1-q)
   = p^{\rank(G)-\rank(X)}(-1)^{\rank(X)}
   \sum_{A\subseteq X}(-1)^{|A|}\chi^*_{X|_A}(q) \chi_{X/A}(p).
\end{equation}
The same statement holds if we instead take  $p \in \mathbb{Z}_M(X)$ and $q \in \mathbb{Z}_A(X)$.
 \end{Corollary}

\begin{Remark}
Suppose that the list $X$ in Corollary \ref{Corollary:FlowsColourings} is the quotient of a scaled unimodular list, \ie it satisfies the following conditions:
\begin{compactenum}
 \item There is a list $X_0 = (x_1,\ldots, x_N) \subseteq \Z^d$ (for some $d,N\in \N$) and $A_0\subseteq X_0$  \st
  $X=X_0/A_0$.
  \item There is a sequence of integers $ ( b_1,\ldots, b_N)$ \st
  the scaled list $ \tilde X_0 := (\frac{1}{b_1} x_1,\ldots, \frac{1}{b_N} x_N)$ is integral and totally unimodular.
\end{compactenum}
 Let $\tilde A_0$ be the subset of $\tilde X_0$ that corresponds to $A_0\subseteq X_0$ and
let $\tilde X:=\tilde X_0/\tilde A_0$. Then 
 $\aritutte_{\tilde X}(x,y) = \tutte_{X}(x,y).$  
 Note that due to total unimodularity, $\tilde X$ is contained in a free abelian group.

Therefore, 
we can interpret the arithmetic Tutte polynomial $\aritutte_X$ in terms of classical flows and arithmetic colorings, or vice versa.
More specifically, 
in the previous corollary we can obtain 
\begin{equation}
\aritutte_{X}(1-p,1-q) = 
p^{\rank(G)-\rank(E)}(-1)^{\rank(E)}\sum_{A\subseteq E}(-1)^{|A|}\chi^*_{X|_A}(q) \chi_{\tilde X/A}(p)
\end{equation} 
for any $p\in \Z$ and $q\in \Z_A(X)$. For $p\in \Z_A(X)$ and any $q\in \Z$ we obtain
 \begin{equation}
\aritutte_{X}(1-p,1-q) = 
p^{\rank(G)-\rank(E)}(-1)^{\rank(E)}\sum_{A\subseteq E}(-1)^{|A|}\chi^*_{\tilde X|_A}(q) \chi_{X/A}(p).
\end{equation} 

 Lists with these properties arise naturally when studying arithmetic matroids defined by labeled graphs
 \cite{moci-adderio-flow-2013}.
 In this case
 $X$ is a list of vectors coming from a labeled graph and $\tilde X$ 
 is the totally unimodular list of vectors that represents the underlying graphic matroid.
 Arithmetic matroids that can be represented by a quotient of a scaled unimodular list
 are studied in more detail in  \cite{lenz-ppcram-2017}. They can be characterized as arithmetic matroids
 that are regular and strongly multiplicative.
\end{Remark}

 \section{Background}
 \label{Section:Background}

\subsection{Matroids and polymatroids}

Let $M$ be a finite set and $\rank : M \to \Z_{\ge 0}$ be a function that satisfies the following axioms:
\begin{compactitem}
\item
 $\rank(\emptyset)=0$,
\item
 $\rank(A)\le \rank(B)$ for all $A\subseteq B\subseteq M$, and
\item 
 $\rank(A \cup B ) + \rank(A \cap B ) \le \rank(A) + \rank( B )$ for all $ A, B \subseteq M$. 
\end{compactitem}
Then the polytope
\begin{equation}
   \left\{ x \in \R^M : 0 \le \sum_{ i \in S} x_i \le \sum_{ i \in S} \rank(x_i) \text{ for all } S\subseteq M \right\}
\end{equation}
 is  called a \emph{discrete polymatroid} and $\rank$ is its rank function  \cite[Chapter~44]{schrijver-co-volB}.

 A \emph{matroid} is a pair $(M,\rank)$, where $M$ denotes a finite set
and the rank function $\rank : 2^M \to \Z_{\ge 0}$ satisfies the axioms of the rank function of a discrete polymatroid
and in addition,  $\rank(A\cup \{a\}) \le \rank(A) +1 $ for all $A\subseteq M$ and $a\in M$ holds.
See \cite{MatroidTheory-Oxley} for more details.
 Let $\K$ be a field.
 A %
 matrix $X$ with entries in $\K$ 
 defines a matroid in a canonical way:
 $M$ is the set of columns of the matrix and the rank function 
 is the rank function from linear algebra.
 A matroid that  can be represented 
 in such a way is called \emph{representable over $\K$}.

\subsection{Arithmetic matroids}
\begin{Definition}[D'Adderio--Moci, Br\"and\'en--Moci
\cite{branden-moci-2014,moci-adderio-2013}]
 An \emph{arithmetic matroid} is a triple $(M, \rank, m)$, where
 $(M, \rank)$ is a matroid and $m : 2^M \to \Z_{\ge 1}$ is the
 \emph{multiplicity function}
 that satisfies certain axioms:
 \begin{itemize}
   \item[(P)]  Let $R\subseteq S\subseteq M$. The set $[R,S]:=\{A: R\subseteq A\subseteq S\}$ is called a \emph{molecule} 
 if $S$ can be written as the disjoint union $S=R\cup F_{RS}\cup T_{RS}$ and for each $A\in [R,S]$,
 $\rank(A)= \rank(R) + \abs{A \cap F_{RS}}$ holds.
 For each molecule   $[R,S] \subseteq M$, the following inequality holds 
 \begin{equation}
  \label{eq:Paxiom}
  \rho(R,S) := (-1)^{\abs{T_{RS}}} \sum_{A\in [R,S]} (-1)^{\abs{S} - \abs{A} } m(A) \ge 0.
 \end{equation}
  \item[(A1)]  For all $A \subseteq M$ and $e \in M$:
if $\rank(A \cup \{e\}) = \rank(A)$, then $m(A \cup \{e\})| m(A)$.
Otherwise $m(A) | m(A \cup \{e\})$.
  \item[(A2)]  If $[R, S]$ is a molecule, then $m(R)m(S) = m(R \cup F ) m(R \cup T )$.
  \end{itemize}
 A \emph{pseudo-arithmetic matroid} is a triple $(M, \rank, m)$, where
 $(M, \rank)$ is a matroid and $m : 2^M \to \R_{\ge 0}$ satisfies (P).
 A \emph{quasi-arithmetic matroid} is a triple $(M, \rank, m)$, where
 $(M, \rank)$ is a matroid and $m : 2^M \to \Z_{\ge 1}$ satisfies (A1) and (A2).
\end{Definition}

 The prototypical example of an arithmetic matroid 
 is defined by a list of vectors $X$ in $\Z^d$.
 In this case, for a sublist   $S$ of $d$ vectors that form a basis, we have $m(S) = \abs{\det(S)}$
 and in general $m(S) := \abs{\left\langle x \in S\right\rangle_\R \cap \Z^d / \left\langle x \in S\right\rangle_\Z}$. 
 As quotients of $\Z^d$ are in general not free groups, the following definition will use a slightly more general setting.

\begin{Definition}
 Let $\Acal = (M, \rank, m)$ be an arithmetic matroid.
 Let $G$ be a finitely generated abelian group and $X $ a finite list of elements of $G$ that is indexed by $M$.
 For $A\subseteq M$, let  $G_A$ denote the maximal subgroup of $G$ \st $ \abs{G_A / \left\langle A \right\rangle}$ is finite.

  $X$ is called a \emph{representation} of $\Acal$ if the matroid defined by $X$ is isomorphic to $(M,\rank)$ 
  and 
  $m(A) = \abs{G_A / \left\langle A \right\rangle}$.
  The arithmetic matroid $\Acal$ is called \emph{representable} if it has a representation $X$.
\end{Definition}

 Given a representation $X\subseteq \Z^d$ of an arithmetic matroid, it is easy to calculate its 
 multiplicity function \cite[p.~344]{moci-adderio-2013}:
 let $A\subseteq X$, then 
        \begin{equation}
 \label{eq:dependentmultiplicity}
 m(A) = \gcd( \{ m(B) : B \subseteq A \text{ and  } \abs{B} = \rank(B) = \rank(A) \}).
 \end{equation}
 If $A$ is independent, then  $m(A)$ is the greatest common divisor of all minors of size $\abs A$ of the matrix $A$
 (cf.~\cite[Theorem~2.2]{stanley-1991}).

 \subsection{Arithmetic matroids defined by labeled graphs}
   \label{Subsection:LabelledGraphs}
   A \emph{labeled graph} is a graph $\Gcal=(V,E)$ together with a map $\ell : E \to \Z_{\ge 1}$.
   The graph $\Gcal$ is allowed to have multiple edges, but no loops.
   The set of edges is partitioned into a set $R$ of \emph{regular edges} and a set $D$ of \emph{dotted edges}.
   Such a graph defines a graphic arithmetic matroid  \cite{moci-adderio-flow-2013}.
   Its definition extends the usual construction of the matrix representation of a graphic matroid
   by the oriented incidence matrix:
   let $ V = \{v_1,\ldots, v_n\}$. We fix an arbitrary orientation $\theta$ of $E$ \st each edge $e\in E$ can be identified with an ordered pair $(v_i,v_j)$.
   To each edge $e = ( v_i , v_j )$, we associate the element $x_e \in \Z^n$ defined as the vector
   whose $i$th coordinate is $-\ell(e)$ and  whose $j$th coordinate is $\ell(e)$.
    Then we define the  list $X_R := (x_e)_{e\in R}$
   and the group
   $ G := \Z^n / \langle \{ x_e : e \in D \} \rangle$. We denote by $\Acal(\Gcal,\ell)$
   the arithmetic matroid represented by the 
     projection of  $X_R$ to $G$.
   The multiplicity function can easily be calculated: for any $A\subseteq R$
     \begin{equation*}
      m(A) = \gcd \bigg( \bigg\{ \prod_{e\in T} \ell(e) : T \text{ maximal independent subset of } A \cup D \bigg\} \bigg)  \text{ holds.}   
     \end{equation*}

\subsection{Delta-matroids and the Bollob\'as--Riordan polynomial}
\label{Subsection:DeltaMatroidsDefinition}
A \emph{delta-matroid} $D$ is a pair $(E,\mathcal{F})$, where  $E$ denotes a finite set and $ \emptyset \neq \mathcal{F}\subseteq 2^E$ satisfies 
the \emph{symmetric exchange axiom}: for all $S,T\in \mathcal{F}$, if there is an element $u\in S\triangle T$, then 
there is an element $v\in S\triangle T$ such that $S\triangle \{u,v\}\in \mathcal{F}$.   
The elements of $\mathcal{F}$ are called \emph{feasible sets}.
If the sets in $\Fcal$ all have the same cardinality, then 
$(E,\Fcal)$ satisfies the basis axioms of a matroid.
Let $D=(E,\Fcal)$ be a delta-matroid
 and let $\mathcal{F}_{\max}$ and $\mathcal{F}_{\min}$ be the set of feasible sets of maximum and minimum cardinality, 
respectively.  
Define $D_{\max}:=(E,\mathcal{F}_{\max})$ and  $D_{\min}:=(E,\mathcal{F}_{\min})$ to be the 
 \emph{upper matroid} and  \emph{lower matroid} for $D$, respectively  \cite{bouchet-1989}.
Let $
\rank_{\max}$ and $\rank_{\min}$ denote the corresponding rank functions.
In  \cite{chun-moffat-noble-rueckriemen-2014}, the following delta-matroid rank function was defined:
$\rho(D) :=  \frac{1}{2}( \rank_{\max}(D)  +   \rank_{\min}(D) )$, 
and
  $\rho(A) :=   \rho(D|_A) \text{ for } A \subseteq E$.
This can be used to define
 the \emph{(2-variable) Bollob\'as-Riordan polynomial} 
\begin{equation} 
\tilde{R}_D(x,y) :=  \sum_{A\subseteq E}  (x-1)^{\rho(E)-\rho(A)}(y-1)^{|A|-\rho(A)}.
\end{equation} 
If $D$ is a matroid, then  $\rho$ is its rank function.
Note that the delta-matroid rank function $\rho$ is different from Bouchet's birank \cite{bouchet-1989}.
  A delta-matroid is even
 if all feasible sets have the same parity. A ribbon graph defines an even delta-matroid if and only if 
it is orientable \cite[Proposition~5.3]{chun-moffat-noble-rueckriemen-2014}.

 \section{Proofs}
 \label{Section:Proofs}

To prove Theorem~\ref{Theorem:ArithmeticConvolution},
we adapt the proof of Kook--Reiner--Stanton \cite{kook-reiner-stanton-1999} to our more general setting.
We first define a convolution product and note some useful lemmas. 
Two ranked sets with multiplicity are isomorphic if there exists a  bijection between the ground sets that preserves the rank and the multiplicity function.
Let $\matroidclass$ be the set of all isomorphism classes of ranked sets with multiplicity, and let $K$ be a commutative ring with $1$.
For any functions $f,g: \matroidclass \to K$, define the convolution
$f\circ g:\matroidclass \to K$ by
\begin{equation}
 (f\circ g)(M) = \sum_{A\subseteq M} f(M|_A) g(M/A).
\end{equation}
\begin{Lemma}
\label{Lemma:associativity}
The convolution
$\circ$ is associative, with identity element $\delta$, where
 \begin{equation}
   \delta(M) := \begin{cases}
                1 & \text{if } M = \emptyset \\
                0 & \text{otherwise}
               \end{cases}.
 \end{equation}
\end{Lemma}
\noindent Note that there are infinitely many ranked sets with multiplicity on the empty set.
\begin{proof}[Proof of Lemma~\ref{Lemma:associativity}]
 It is easy to see that $\delta$ is the identity element.
 
 Let $C,D \subseteq M$ and $C\cap D=\emptyset$. As in the case of matroids, $ (M / C)|_D = M|_{C\cup D}/C  $, holds:
 let $A\subseteq D$. Then by definition $ \rank_{(M / C)|_D}(A)  = \rank_M(A \cup C) - \rank_M(C) = \rank_{M|_{C\cup D}/C}(A) $.
 For the multiplicity function  by definition
 $ m_{(M / C)|_D}(A) = m_M(C \cup A) = m_{M|_{C\cup D}/C}(A)  $.

 Now let $f,g,h : \matroidclass \to K$.
 \begin{align}
   (( f \circ g ) \circ h ) (M) &=
    \sum_{ A  \subseteq  M  }  (f\circ g)( M|_A ) h( M / A)
  \\
    &=
    \sum_{  A  \subseteq  M  } \sum_{C\subseteq A}    f(M|_C) g(M|_A/C) h( M / A)
  \\
    &=
    \sum_{  C \subseteq A  \subseteq  M  }     f(M|_C) g(M|_A/C) h( M / A)
 \intertext{Now let $D:=A\setminus C$. Hence $A=C\cup D$ and we obtain}
    &\quad
    \sum_{  \substack{ C,D \subseteq  M \\ C \cap D = \emptyset } }      f(M|_C) g(M|_{C\cup D}/C) h( M / (C \cup D))
   \displaybreak[2]
   \\
    &=
      \sum_{ C \subseteq M  }  f  ( M|_C )  \sum_{ D \subseteq M \setminus C}   g((M/C)|_D ) h( (M/C) / D)  
\displaybreak[2]
     \\
    &=
      \sum_{ C \subseteq M  }  f  ( M|_C )  ( (g \circ h)( M / C)  )
   \\
  &=
   ( f \circ  ( g  \circ h )) (M). \notag\qedhere
 \end{align}

\end{proof}

Following Crapo \cite{crapo-1969}, let 
$\zeta(x,y)(M) %
:= x^{\rank(M)}y^{\abs{M} - \rank(M)}$, where $K = R[x,y]$.
The following simple lemma was proven for matroids in \cite{kook-reiner-stanton-1999}.
It is easy to verify that the same proof also works in our setting. Here, $\rank(\emptyset)=0$ is required.
\begin{Lemma}
\label{Lemma:zetaInverse}
 $\zeta(x,y)^{-1} = \zeta(-x,-y)$.
\end{Lemma}
Note that $\zeta$ only depends on the matroid, but not on the multiplicity function $m$.
We will also need two weighted versions of 
$\zeta$, namely
\begin{align}
\xi(x,y)(M) &:= m_M(M) x^{\rank(M)}y^{\abs{M}-\rank(M)} \\
\text{and }
\xi^*(x,y)(M) &:=   m_M(\emptyset) x^{\rank(M)}y^{\abs{M}-\rank(M)}.
\end{align}
If $M$ is an arithmetic matroid, then 
 $\xi^{*}(M)=  m^*_{M}(M) x^{\rank(M)}y^{\rank(M^*)}$,
 since
 $\rank(M^*)=\abs{M} - \rank(M)$  and
 the dual multiplicity is defined by $ m^*(A) := m( M \setminus A ) $.

The following well-known 
description of the Tutte polynomial \cite{crapo-1969} generalizes to our setting. 
\begin{Lemma}
\label{Lemma:TutteZetaConvolution}
 \begin{align}
  \tutte_M(x+1,y+1) &=
   (\zeta(1,y) \circ \zeta(x,1)) (M) 
 \end{align}
\end{Lemma}
Lemma~\ref{Lemma:TutteZetaConvolution} is actually a special case ($m\equiv 1)$ of the next lemma.
\begin{Lemma}
\label{Lemma:AritutteConvolution}
 \begin{align}
  \aritutte_M(x+1,y+1) &=
   (\xi(1,y) \circ \zeta(x,1)) (M) =   (\zeta(1,y) \circ \xi^*(x,1)) (M)
 \end{align}
\end{Lemma}
\begin{proof}
  \begin{align}
      (\xi(1,y) \circ \zeta(x,1)) (M) 
       &= \sum_{ A \subseteq M  } m_M(A) y^{\abs{(M|_A)} - \rank(M|_A)}  x^{\rank(M/A)}
       \\
       &=  
       \sum_{ A \subseteq M  } m_M(A)  x^{\rank(M) - \rank(A)} y^{ \abs{A} -\rank(A)}
       = \aritutte_M(x+1, y+1)
       \\
      (\zeta(1,y) \circ \xi^*(x,1)) (M) 
       &= \sum_{ A \subseteq M  }  y^{\abs{(M|_A)} - \rank( M|_A )}  m_{(M/A)}( \emptyset   )  x^{\rank(M/A)}
       \\
       &=  
       \sum_{ A \subseteq M  } m_M(A)  x^{\rank(M) - \rank(A)} y^{ \abs{A} -\rank(A)}
       = \aritutte_M(x+1,y+1). \notag
  \end{align}
\end{proof}

\begin{proof}[Proof of Theorem~\ref{Theorem:ArithmeticConvolution}]
   Lemma~\ref{Lemma:AritutteConvolution} implies
   \begin{align}
    \aritutte_M(x+1,0) &=   %
    ( \zeta( 1, -1) \circ \xi^* (x, 1) ) (M) \\
  \text{and }  \aritutte_M(0,y+1) &=    ( \xi(1, y)  \circ \zeta ( -1, 1) ) (M). 
   \end{align}
Using Lemma~\ref{Lemma:zetaInverse} and Lemma~\ref{Lemma:TutteZetaConvolution} we obtain
  \begin{align}
   \sum_{A\subseteq M} & \aritutte_{M|_A}(0, y + 1) \tutte_{M/A}(x + 1,0) \\
   &= ( ( \xi(1, y)  \circ \zeta ( -1, 1) ) \circ ( \zeta( 1, -1) \circ \zeta (x, 1) ) ) (M) \\
   &= (  \xi(1, y)  \circ (\zeta ( -1, 1)  \circ   \zeta( 1, -1) )\circ \zeta (x, 1)  ) (M)  \\
   &= (  \xi(1, y)  \circ  \zeta (x, 1)  ) (M) 
   = \aritutte(x+1, y+1)  \\
\intertext{ and}
   \sum_{A\subseteq M}  &\tutte_{M|_A}(0, y+1) \aritutte_{M/A}(x+1,0) \\
   &= ( ( \zeta(1, y)  \circ \zeta ( -1, 1) ) \circ ( \zeta( 1, -1) \circ \xi^* (x, 1) ) ) (M) \\
   &= (  \zeta(1, y)  \circ (\zeta ( -1, 1)  \circ   \zeta( 1, -1) )\circ \xi^* (x, 1)  ) (M) \\
   &= (  \zeta(1, y)  \circ  \xi^* (x, 1)  ) (M)  
    =\notag \aritutte(x+1, y+1).\qedhere
   \end{align}
   
\end{proof}

\begin{proof}[Proof of Corollary~\ref{Corollary:PositivityOfCoefficients}]
 Using Theorem~\ref{Theorem:ArithmeticConvolution} twice, we obtain
 \begin{align}
   \aritutte_M(x,y) &= 
   \sum_{A\subseteq M}  \aritutte_{M|_A}(0, y) \tutte_{M/A}(x,0)
  \\ &= \sum_{A\subseteq M} \tutte_{M/A}(x,0)  \left(\sum_{B\subseteq A}  \tutte_{M|_B}(0, y) \aritutte_{M|_A /  B }(0,0) \right).
 \end{align}

 Hence it   is sufficient to show that   
 \begin{equation}
 \aritutte_M(0,0) =  \sum_{A\subseteq M} (-1)^{\rank(M) - \abs{A}} m(A) \ge 0
 \end{equation}
 for 
 any pseudo-arithmetic matroid $M$.
 This can be shown in various ways: 
 \begin{asparaenum}[(i)]
  \item induction.
  \item It is well-known that $2^M$ can be partitioned into molecules $[R,S]$ with $\rank(R)=\abs{R}$ and $\rank(S)=\rank(M)$ 
  (\cite[Proposition 4.4]{branden-moci-2014}, see also \cite{bjoerner-1992,crapo-1969}). For each such molecule
  $\rank(M) = \abs{R} + \abs{F_{RS}} = \abs{S\setminus T_{RS}}$ holds. Hence we obtain
  \begin{equation}
     \aritutte_M(0,0) = \sum_{\substack{[R,S]\\\text{molecule}}} \underbrace{(-1)^{\abs{T_{RS}}}\sum_{R\subseteq A\subseteq S} (-1)^{\abs{S} - \abs{A}} m(A)}_{\rho(R,S)\ge 0} \ge 0.
  \end{equation}
  \item In the case of an arithmetic matroid that is represented by a list of vectors it  follows from the interpretation of
  $\aritutte(0,q)$ in \cite{lenz-arithmetic-2016}. \qedhere
 \end{asparaenum}
 \end{proof}

\begin{proof}[Proof of Remark~\ref{Remark:toricvertexdecomposition}]
   We will now prove that 
   if we set $x=1$, the second expression for $\aritutte_M(x,y)$ in Theorem~\ref{Theorem:ArithmeticConvolution} is equivalent to 
   \cite[Lemma~6.1]{moci-tutte-2012}.

 Using \cite[Lemma~6.1]{moci-tutte-2012} and the classical convolution formula we obtain
 \begin{align}
  \aritutte_M(1,y) &=   \sum_{p\in \Vcal(X)} \tutte_{M_p}(1,y) \\
  &= \sum_{p \in \Vcal(X)} \sum_{A\subseteq M_p} \tutte_{(M_p)|_A}(0, y) \tutte_{M_p/A}(1, 0) \\
  &= \sum_{A\subseteq M}   \tutte_{M|_A}(0, y) \:  { \left( \sum_{ p : A \subseteq M_p }  \tutte_{M_p/A}(1, 0) \right) }    \label{eq:interesting}\\ 
  &= \sum_{A\subseteq M}   \tutte_{M|_A}(0, y)  {\left(\sum_{ \bar p \in \Vcal(M/A) }  \tutte_{(M/A)_{\bar p}}(1, 0)\right)} 
  \label{eq:interestingB}
\\
  &= \sum_{A\subseteq M}   \tutte_{M|_A}(0, y)    \aritutte_{M/A} (1,0).
 \end{align}
 Recall that the vertices of the generalized toric arrangement are contained in the generalized real torus 
 $\hom(G, S^1)$, where $G$ denotes a finitely generated abelian group.
  To verify the equality of \eqref{eq:interesting} and
 \eqref{eq:interestingB}, note that
 $\{p\in \Vcal(X)\subseteq \hom(G, S^1) : A\subseteq M_p\} = \{p\in \Vcal(X) : p(A)= \{1\} \} \leftrightarrow 
  \Vcal(X/A) \subseteq \hom( G / \langle A\rangle, S^1)$ 
 and $ M_p / A = (M/A)_{\bar p}$ since restriction and contraction commute.
  We have also used that since $A\subseteq M_p$, $\tutte_{M_p|A}(x,y)  = \tutte_{M|A}(x,y)$.
   
  For the other direction, note that $ m(A) = \abs{\{p\in \Vcal(X) : A\subseteq X_p\}}$ holds by \cite[Lemma~5.4]{moci-tutte-2012}.
   Hence 
   \begin{equation}
   \aritutte_X(1,0)
   = \sum_{p\in \Vcal(X)}   \sum_{\substack{A\subseteq X_p\\\rank(A)=\rank(X)}}
   (-1)^{\abs A -\rank(A)}  
   = \sum_{ p\in \Vcal(X) } \tutte_{X_p}(1,0).
   \end{equation}
   Now \cite[Lemma~6.1]{moci-tutte-2012} follows using essentially the same calculation as above.
\end{proof}

\begin{proof}[Proof of Corollary~\ref{Corollary:PforProduct}]
 Let $[R,S]$ be a molecule. We need to show that $\rho(R,S)$ is nonnegative for the multiplicity function $m_1m_2$.
 Note that the positivity axiom is closed under minors: for deletions it is obvious and for contractions it follows from the fact that
  $[R,S]$ is a molecule in the contraction  $M/e$ if and only if  $[R \cup \{e\}, S\cup \{e\}]$ is a molecule in $M$. 
 
 It is known that $\rho(R,S)$ is the 
 constant coefficient of the arithmetic Tutte polynomial
 obtained by restricting to $S$ and contracting 
 the elements in $R$. This was observed in the proof of 
 \cite[Lemma 4.5]{branden-moci-2014} using  \cite[Lemma 4.3]{branden-moci-2014}\footnote{Note that 
 \cite[Lemma 4.5]{branden-moci-2014} contains a small error: the factor $(y-1)^{\abs R - \rank(R)}$ is missing on the right-hand side of the first equation.
}.

Hence by Theorem~\ref{Theorem:GeneralisedMainTheorem} and Corollary~\ref{Corollary:PositivityOfCoefficients}
\begin{align}
 \rho(R,S) &= \aritutte_{((M,\rank,m_1  m_2)|_S)/R}(0,0) 
     \\
     &= \sum_{A\subseteq S \setminus R}  \underbrace{\aritutte_{ ((M,\rank,m_1)/R)   |_A }(0, 0)}_{\ge 0} \underbrace{\aritutte_{(M,\rank,m_2) / ( R \cup A)}(0,0)}_{\ge 0} \ge 0.
     \notag\qedhere
\end{align}

\end{proof}
\begin{proof}[Proof of Corollary~\ref{Corollary:TwoOne}]
It is known that
$\abs{Z(X)\cap \Z^d} = \aritutte(2,1)$ and 
$\bigl|\relint Z(X) \cap \Z^d\big| = \aritutte(0,1)$ \cite{moci-adderio-ehrhart-2012,stanley-1980,stanley-1991}.
The second equality is Theorem~\ref{Theorem:ArithmeticConvolution}.
The third follows from the fact that $\tutte_{M/A}(2,0)=0$ if $A$ is not a flat since in this case, $M/A$ contains a loop. 
Furthermore, the number of vertices of the zonotope is equal to the number of regions 
of the central hyperplane arrangement defined by $X$
\cite[Proposition~2.2.2]{OrientedMatroidsBook}. This number equals 
$\tutte_{M}(2,0)$ \cite{zaslavsky-1975}.
For a flat $A$, there is a canonical bijection between the vertices of $Z(X/A)$ and the faces of $Z(X)$ that correspond to $A$.
\end{proof}

\begin{proof}[Proof of Corollary~\ref{Corollary:squared}]
\begin{align}
 \aritutte_{X^2} (1-p,1-q)
 &= \sum_{A\subseteq X}  \aritutte_{X |_A}(0, 1-q) \aritutte_{X/A}(1-p,0)
\\&=\sum_{A \subseteq X} (-1)^{|A|-\rank(A)}\chi^*_{X|_{A}}(q)(-1)^{\rank(X/A)}p^{\rank(G/A)-\rank(X/A)}\chi_{X/A}(p) 
\\
\notag
 &=p^{\rank(G)-\rank(X)}\sum_{A \subset X} (-1)^{|A|-\rank(A)+\rank(X/A)}\chi^*_{X|_{A}}(q)\chi_{X/A}(p) 
  \\
&=p^{\rank(G)-\rank(X)}(-1)^{\rank(X)}\sum_{A\subseteq X}(-1)^{|A|}\chi^*_{X|_A}(q) \chi_{X/A}(p) .    
\end{align}
The first two steps use Theorems \ref{Theorem:GeneralisedMainTheorem} and~\ref{Theorem:BrandenMociFlowsColourings}.
The third uses
$\rank(X/A)-\rank(G/A) = \rank(X)-\rank(G)$. The last equality holds
because $(-1)^{\rank(X/A)-\rank(A)}=(-1)^{\rank(X)}$.
\end{proof}

\begin{proof}[Proof of Corollary~\ref{Corollary:FlowsColourings}]
This follows by the same argument as in the proof of Corollary \ref{Corollary:squared}, using
 Theorem~\ref{Theorem:ArithmeticConvolution} instead of
 Theorem~\ref{Theorem:GeneralisedMainTheorem} in the first step.
\end{proof}

\subsection*{Acknowledgements} We would like to thank Petter Br\"and\'en, Emanuele Delucchi, Alex Fink,
  Emeric Gioan,  Katharina Jochemko,
  Iain Moffat, and Mich\`ele Vergne for helpful comments and interesting discussions.
 We express our gratitude to the organizers of
 the joint session of the  
 \textit{Incontro Italiano di Combinatoria Algebrica} and the 
 \textit{S\'eminaire Lotharingien de Combinatoire}
 2015 in Bertinoro and the organizers of the \textit{Borel Seminar on Matroids in Algebra, 
 Representation  Theory and Topology} 2016 in Les Diablerets for providing
 a hospitable working 
 environment, where some of this work was carried out.
 The first author would like to thank the mathematics department of the \emph{Universit\'e de Fribourg} for the hospitality
 during his visit in November 2015.
  
  An extended abstract of this paper will appear in the proceedings of FPSAC 2017 
  (29th International Conference on Formal Power Series and Algebraic Combinatorics)
  \cite{backman-lenz-fpsac-2017}.

\renewcommand{\MR}[1]{} 

\bibliographystyle{amsplain}
\input{ArithmeticConvolution_v3.bbl}

\end{document}

%% file: ArithmeticConvolution_v3.bbl
\providecommand{\bysame}{\leavevmode\hbox to3em{\hrulefill}\thinspace}
\providecommand{\MR}{\relax\ifhmode\unskip\space\fi MR }
\providecommand{\MRhref}[2]{%
  \href{http://www.ams.org/mathscinet-getitem?mr=#1}{#2}
}
\providecommand{\href}[2]{#2}